\documentclass[graybox,envcountsame,envcountsect]{svmult}
\usepackage{mathptmx}       
\usepackage{helvet}         
\usepackage{courier}        
\usepackage{type1cm}  
\usepackage{amscd,amsmath,amssymb}
\usepackage{graphicx}
\usepackage[all]{xy}
\usepackage{makeidx}         
\usepackage{multicol}        
\usepackage[bottom]{footmisc}
\usepackage{url}
\smartqed

\makeindex             

\newcommand{\Hom}{\operatorname{Hom}\nolimits}
\newcommand{\Ext}{\operatorname{Ext}\nolimits}

\renewcommand{\c}[1]{\mathcal #1}
\newtheorem{questions}[theorem]{Questions}

\newcommand{\gldim}{\operatorname{gldim}}

\renewcommand{\sup}{\operatorname{sup}}
\renewcommand{\Ext}{\operatorname{Ext}}
\newcommand{\dm}{\operatorname{dim_k}\nolimits}
\def\m{\text{mod}\,\Lambda}

\begin{document}

\title*{Algebras of finite global dimension}
\titlerunning{Algebras of finite global dimension}
\author{Dieter Happel and Dan Zacharia}
\authorrunning{Happel and Zacharia}
\institute{Dieter Happel \at Fakult\"at f\"ur Mathematik, Technische Universit\"at 
Chemnitz, 09107 Chemnitz, Germany., \email{happel@mathematik.tu-chemnitz.de}
\and Dan Zacharia \at Department of Mathematics, Syracuse University,
Syracuse, NY 13244-0001, USA. \email{ zacharia@syr.edu}}


\thanks{The second author is supported by NSA grant H98230-11-1-0152. Most of the 
material presented here results  from numerous discussions over the years 
during mutual visits. It contains 
some of the material 
presented by the first author during a talk at the 2011 Abel Symposium in 
Balestrand, 
Norway. Both authors are  thankful for the invitation to participate at this 
symposium.}


\maketitle

\abstract{Let $\Lambda$ be a finite dimensional algebra over an algebraically closed 
field $k$. We survey some results on algebras of finite global dimension
and address some open problems.}

\bigskip
\bigskip

Let $\Lambda$ be a finite dimensional algebra over an algebraically closed
field $k$. We denote
by $\m$ the category of finitely generated left $\Lambda$-modules. 
In this article we are mainly interested in algebras of finite global
dimension, so each $X\in\m$ admits a finite projective resolution, or
equivalently each simple $\Lambda$-module admits a finite projective 
resolution. It is well-known that the global dimension, $\gldim\Lambda$, of $\Lambda$ is 
the maximum 
of the lengths of these 
finitely many minimal projective resolutions of the simple $\Lambda$-modules.
These notions go back to the pioneering book by Cartan and Eilenberg \cite{CE}.
They were intensively studied in a famous series of ten papers in the 
Nagoya Journal written by various authors and published over the years 1955 
to 1958, see \cite{A1}, \cite{A2}, \cite{B}, \cite{EN1}, \cite{EN2}, \cite{EIN},  \cite{ENN},
 \cite{ERZ}, \cite{JN} and \cite{Ka}
for this series of articles.

\medskip

The global dimension being preserved under Morita equivalence, implies that we may assume without loss of generality that 
$\Lambda$ is basic.  As the field $k$ is algebraically closed, $\Lambda$ 
is given by a quiver with relations. We briefly recall the construction. 
We start by recalling the definition of an admissible ideal. Let $Q$ be
a finite quiver and let $kQ$ be the path algebra over $k$. Recall that the set  $\mathcal W=\{\text{paths in}\,\,Q\}$ forms a $k-$basis for 
$kQ$. Denote by
$Q_{\ge t}$ the two sided ideal of $kQ$ generated by all the paths in $Q$ of length
$t$. A two sided ideal $I$ in $kQ$ is called {\it admissible} if there exists a natural number $t\ge 2$ such that 
$$Q_{\ge t}\subseteq
I\subseteq Q_{\ge 2}.$$ 

\noindent
Then it is well-known (see \cite{G}) that every basic finite dimensional $k$-algebra $\Lambda$ 
satisfies $\Lambda\simeq kQ/I$ for some finite quiver $Q$ and admissible ideal $I$ in $kQ$. 
Note that the quiver $Q$ is uniquely determined
by $\Lambda$. By abuse of language, a minimal generating set of the ideal $I$ is called
a set of relations for $\Lambda$. This set is always finite, but
we may have different choices for the relations for $\Lambda$. In 
particular there is usually no canonical choice for the relations.
Note also that due to our assumptions on the admissible ideal $I$, the quiver $Q$ can be 
recovered from $\Lambda=kQ/I$. To be more precise, the vertices of $Q$ correspond to 
the (isomorphism classes of) simple
$\Lambda$-modules and the number of arrows from a simple 
$\Lambda$-module $S$
 to a
simple $\Lambda$-module $S'$
coincides with $\dm\Ext^1_\Lambda(S,S')$. We remind the reader 
that there is also a more ring theoretic version to find the quiver of 
$\Lambda$
(see for example \cite{ARS}).

\medskip

We will address two basic questions for algebras of finite global dimension.
First we deal with the question on possible obstructions for the quiver
of an algebra of finite global dimension. It has been known for a long time
\cite{L}, \cite{I} that the quiver of an algebra $\Lambda$ of finite global 
dimension does not contain any loops, or equivalently  
$\Ext^1_\Lambda(S,S)=0$ for
all simple $\Lambda-$modules $S$.
This result has been recently strengthened by \cite{ILP}. For details see Section 
\ref{obstr}. 

\medskip

There is a completely different behavior for algebras $\Lambda$ of finite
global dimension satisfying
$\gldim\Lambda\le 1$ or $\gldim\Lambda\ge 2$.

\medskip

In the case when $\gldim\Lambda\le 1$, the algebra $\Lambda$ 
is hereditary. Thus $\Lambda=kQ$ for some finite quiver $Q$. Since we assume that $\Lambda$ is 
finite dimensional,
this implies that the quiver $Q$ has no oriented cycles. Conversely, if 
$Q$
does not contain any oriented cycle, then the algebra $kQ$ satisfies
$\gldim\,kQ\le 1$. Of course $\gldim\,kQ = 0$ if 
and only if $Q$ contains no arrows. So this case is well-understood. 

\medskip

In contrast to the first case the situation becomes more complicated and 
interesting for higher values of $\gldim\,\Lambda$.
Moreover there are still a lot of open problems to deal with in this case. 
We will show in Section \ref{constr} that for an arbitrary quiver 
there will
exist a two sided ideal $I$ in the path algebra $kQ$ such that for 
$\Lambda=kQ/I$ we have that
$\gldim\,\Lambda\le 2$, as long as the quiver does not contain a 
loop. This had been known for a long time by Dlab and Ringel \cite{DR1} and 
\cite{DR2}. It was recently rediscovered independently by N. Poettering 
\cite{P} and the first author. We are grateful to C. M. Ringel for pointing 
out the two references mentioned above. We will also prove another related 
result in this 
section and pose a few open problems. For example, obtaining necessary and sufficient conditions on a given quiver $Q$ such that there exists
a two sided ideal $I$ with $\gldim\,kQ/I=d$ for a prescribed 
natural number $d\ge 3$ is still open.

\medskip

In Section \ref{bounds} we address the following two related
problems. Fix a quiver $Q$ without loops, and consider the 
following set of algebras
$$\mathcal A(Q)=\{kQ/I\,|\,\text{dim}_kkQ/I<\infty\,\,\text{and}
\,\,\gldim\,kQ/I<\infty\}.$$

\noindent
We then define the following 
$$d(Q)=\text{sup}\,\{\dm kQ/I\,|\,kQ/I\in
\mathcal A(Q)\}$$

\noindent and 
$$g(Q)=\text{sup}\,\{\gldim kQ/I\,|\,kQ/I\in
\mathcal A(Q)\}.$$

\noindent
The basic problem is whether or not these are finite. In general 
this seems to be an open problem. Note that the set $\mathcal A(Q)$
can be infinite. A concrete example is included in Section \ref{bounds}.

\medskip

We will look at the relationship between $d(Q)$ and $g(Q)$. 
For example it follows from 
a result in \cite{S} that $g(Q)<\infty$ if $d(Q)<\infty$.
We will also discuss what is known about  $d(Q)$ and $g(Q)$ when one restricts the 
algebras in $\mathcal A(Q)$ to some special subclasses of algebras. These subclasses include the
serial, monomial and quasi-hereditary algebras. 
We remind the reader that except for a few instances,
the global dimension of an algebra $\Lambda$ 
does not depend solely on the number of simple $\Lambda$-modules, nor 
on the Loewy length of the algebra $\Lambda$. We refer to Section \ref{bounds} for details
and some examples.

\medskip

We are aware of quite a number of  other problems and results
for algebras of finite global dimension.
For example we could mention the Cartan determinant problem (see for example
\cite{Z1}). Also we will not deal with homologically finite subcategories
in module categories for algebras of finite global dimension (see for example
\cite{AR1}, \cite{AR2} or \cite{HU}). Moreover, in order to keep the level of 
exposition as simple as possible, we will not attempt to formulate the results 
in the most general form available and will not use the language of derived 
categories (see for example \cite{H1}).
Lastly, we will not attempt to be complete. We rather concentrate on the 
things which were mentioned before.

\medskip

We denote the composition of morphisms $f\colon X\to Y$ and $g\colon Y\to Z$ in a 
given category ${\c K}$ by $fg$. The notation and terminology introduced here will be fixed throughout
this article. For unexplained representation-theoretic terminology, 
we refer to \cite{ARS}, and \cite{R1}.  

\section{Preliminaries}
\label{prelim}

In this section we recall some further notation and definitions 
that will be used in the paper. Keeping the notation from the introduction, we assume that the finite dimensional algebra $\Lambda$ is isomorphic to $kQ/I$ for a finite quiver $Q$ and an admissible two sided
ideal $I$ in the path algebra $kQ$. The set of vertices $Q_0$ of $Q$
will be identified with the set $\{1,2,\dots,n\}$. If $\alpha$ is an arrow
of $Q$, we denote by $s(\alpha)\in Q_0$ and by $e(\alpha)\in Q_0$ 
the starting point of $\alpha$ (end point respectively). A path 
$w=(i|\alpha_1,\dots,\alpha_r|j)$ from $i$ to $j$ in $Q$ is by definition   
a sequence of consecutive arrows $\alpha_1,\dots,\alpha_r$ such that $s(\alpha_1)=i$,
$e(\alpha_t)=s(\alpha_{t+1})$, for $1\le t<r$ and $e(\alpha_r)=j$.
The corresponding element $w$ in $kQ$ is denoted by $\alpha_1\dots\alpha_r$ and
we say that this path has length $r$.
For a vertex $i$ of $Q$ we denote by $S(i)$ the corresponding
simple $\Lambda-$module and by $P(i)$ its projective cover. If $X\in\m$
we denote by $\text{proj.dim}_\Lambda X$ the projective dimension of $X$, that is, 
the length $d$ of a minimal projective resolution of $X$:
$$0\to P^d\to P^{d-1}\to\cdots\to P^1\to P^0\to X\to 0$$
\noindent
Let us pause for a moment. Since in this article we are mainly interested in 
algebras of finite global dimension, we want to point out here a few additional
facts. 

First, assume that $\Lambda$ is
given as $kQ/I$  and that its quiver  $Q$ contains no oriented cycle. In this case 
the global dimension of $\Lambda$ is always finite and is bounded by the length of a maximal path in $Q$. In particular, $\gldim\Lambda\le 
|Q_0|-1$.
So the more interesting algebras for us in this context are 
those whose
quivers contain oriented cycles. By glueing techniques one could even assume
that each vertex of $Q$ lies on an oriented cycle. But we refrain from
going into this here.

\medskip

We should also make another remark which is useful
to keep in mind, but will not be further investigated in this paper. Suppose that we 
have given a 
quiver $Q$
and an ideal $I$ with generating set $r_1,\dots, r_t$ having the property 
that the 
coefficients
occurring in $r_j$ only belong to $\{1,-1\}$ for all $1\le j\le t.$ Then we 
may consider
$\Lambda_k=kQ/I$ for different fields $k.$ It is well-known that 
$\gldim\Lambda_k$
may depend on the characteristic of the field $k.$ Examples for this 
phenomena can be found for instance
in \cite{C} and \cite{IZ2}.
They were obtained in the following way. Start with a simplicial complex 
and consider the 
induced partial
order. Then consider the corresponding incidence algebra. Then simplicial 
homology is 
related to the 
Hochschild cohomology of the incidence algebra \cite{GS}. There are well 
known examples 
that simplicial homology depends
on the characteristic of the field of coefficients. This is used in 
\cite{IZ2} to construct 
examples where
the global dimension depends on the characteristic of the field.  For 
related investigations 
we also refer to \cite{GSZ}.

\medskip
For later purposes
we will need the construction of the {\it standard modules} $\Delta(i)$ for
$i\in Q_0.$ For this let $\le$ be a fixed partial order on the set 
of vertices $Q_0=\{1,2,\dots,n\}$.
For a vertex $i\in Q_0$ the module $\Delta(i)$ is the largest quotient of
$P(i)$ with composition factors $S(j)$ for $j\ge i$. We say (see \cite{R2}) that $\Lambda$ is
{\it strongly quasi-hereditary} (with respect to the partial order $\le$) 
 if for every $i\in Q_0$ there is an 
exact sequence 
$$0\to \Omega(\Delta(i))\to P(i)\to\Delta(i)\to 0$$
\noindent such that $\Omega(\Delta(i))$ is a direct sum of projective modules $P(j)$
with $j<i$ and $ \text{End}_\Lambda \Delta(i)\simeq k$ for all $1\le i\le n$.

\medskip
It was shown in \cite{R2} that a strongly quasi-hereditary algebra is
quasi-hereditary in the usual sense (compare for example \cite{CPS} or
\cite{DR3}) and that
$\gldim\Lambda\le n$, where $n$ is the number of simple pairwise 
different
$\Lambda$-modules.  We refer to \ref{DR} for a bound on the global
dimension for an arbitrary quasi-hereditary algebra $\Lambda$.

\medskip

We will also need the concept of a monomial algebra. They are a good source of
examples, since most homological problems are easy to decide.
At the same time we would like to warn the reader that the monomial algebras are far from being typical amongst
finite dimensional algebras. Recall that a basic
finite dimensional $k$-algebra $\Lambda$ is called a {\it monomial algebra} if
$\Lambda\simeq kQ/I$ for a quiver $Q$ and an admissible ideal $I$
which can be generated by paths in $Q$. We point out that there is still
no ring theoretic characterization for monomial algebras.
 The easiest examples of monomial algebras are the
finite dimensional algebras 
whose quiver is a tree.  For special homological properties of 
monomial algebras like the computations of minimal projective
resolutions we refer to \cite{GHZ} and \cite{HZ}. Note also that for monomial algebras, the global dimension does not depend on the ground field. For a path $w\in Q$ we 
denote by $\bar w$ the residue class of $w$ in $kQ/I$. We have the following straightforward result whose proof is left to the reader. 

\begin{proposition}\label{mono1} Let $\Lambda =kQ/{\langle w_1,\dots,w_r\rangle}$ be a 
finite dimensional monomial algebra.
Then the set $\mathcal J=\{\bar v\in\Lambda\,|\, v\,\,\text{a path 
in}\,\,Q,\,v\notin I \}$ is finite and forms a $k$-basis
of $\Lambda$.\qed
\end{proposition}


\medskip
We recall the concept of a Nakayama algebra or generalized uniserial 
algebra. There are various possible 
definitions. We will use a quite concrete one. For different but equivalent
formulations we refer for example to \cite{ARS} or \cite{Ku}, \cite{N}. 
Let $Q$ be either a linearly oriented
quiver with underlying graph $\mathbb A_n$ or a cycle $\tilde{\mathbb A}_n$
with cyclic orientation. So $Q$ is one of the following

\begin{center}

\includegraphics[scale=0.8]{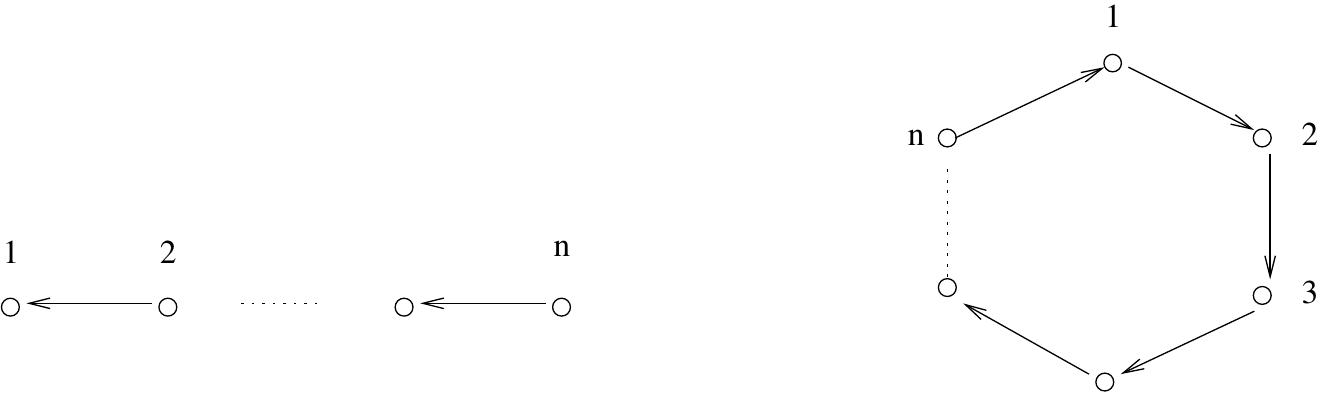}

\end{center}

\noindent
A quotient $\Lambda$ of $kQ$ by an 
admissible ideal is called a Nakayama algebra. Note that a Nakayama algebra 
is a special case of a monomial
algebra. 

\medskip

In the first case of a linearly oriented $\mathbb A_n$ the 
corresponding
Nakayama
algebra $\Lambda$ satisfies $\gldim\Lambda\le n-1,$ so is always 
of finite global 
dimension, while in the second case of
a cycle the corresponding Nakayama algebra may or may not be of finite 
global dimension, 
as very easy examples show. For example, let $Q$ is a cyclic quiver with
two vertices and arrows $\alpha,\beta$. If we choose $I=\langle\alpha\beta\rangle$,
then  $\gldim\,kQ/I =2$. But if we choose 
$I'=\langle\alpha\beta,\beta\alpha\rangle$, then  $\gldim\,kQ/I' =\infty.$
And of course both algebras are Nakayama algebras. 

\medskip

Given a Nakayama algebra $\Lambda$, one may associate to it its 
{\it Kupisch series}. From this series we can decide in a purely 
combinatorial way whether or not $\gldim\Lambda<\infty.$ For
details we refer to \cite{F}.

\section{Obstructions}\label{obstr}

First we will recall the definition of the Hochschild homology and cohomology 
of a finite
dimensional algebra $\Lambda$. We will not use the original definitions due 
to Hochschild \cite{Ho}, but use instead the characterization given in 
\cite{CE}. For this
let $\Lambda^e=\Lambda\otimes_k\Lambda^{op}$ be the enveloping algebra, 
where we
have denoted by $\Lambda^{op}$ the opposite algebra. Then the
$\Lambda-\Lambda$-bimodule ${}_\Lambda\Lambda_\Lambda$ is considered as 
a $\Lambda^e$-module.
For $i\ge 0$ the $i$-th Hochschild homology space is denoted by 
$H_i(\Lambda)$ and is defined as $\text{Tor}_i^{\Lambda^e}(\Lambda,\Lambda)$.
We denote by $[\Lambda,\Lambda]$ the $k$-subspace of $\Lambda$ generated by the
commutators, thus elements of the form $\lambda\mu - \mu\lambda$ for 
$\lambda,\mu\in\Lambda.$ Then it is easy to see that $H_0(\Lambda)$ coincides
with the factor space $\Lambda/[\Lambda,\Lambda]$. 

\medskip

Also for $i\ge 0$ we denote by $H^i(\Lambda)$ the $i$-th Hochschild
cohomology space and this is defined as 
$\Ext^i_{\Lambda^e}(\Lambda,\Lambda)$. So $H^0(\Lambda)$ coincides
with the center of $\Lambda$ and $H^1(\Lambda)$ is the factor space of
all derivations on $\Lambda$ by the subspace of inner derivations.

\medskip

The following theorem is due to Keller \cite{K}.
\begin{theorem}\label{Hochschild} Let $\Lambda$ be a finite dimensional 
algebra of finite 
global dimension. Then $H_i(\Lambda)=0$ for $i>0$ and $H_0(\Lambda)=
k^{|Q_0|}$. In particular,  
$\Lambda/[\Lambda,\Lambda] = k^{|Q_0|}.$
\end{theorem}

We remark that in \cite{Ha} it is conjectured that the converse of 
\ref{Hochschild}
holds. There this converse is proved in special cases such as monomial 
algebras. 

\medskip

We point out that there is no characterization of algebras of finite
global dimension through  Hochschild cohomology. It is easy to see 
that for an algebra
$\Lambda$ of finite global dimension $d$ the Hochschild cohomology spaces
$H^j(\Lambda)=0$ for $j>d$ (see for example \cite{H3}), in particular 
the cohomology algebra $H^*(\Lambda)=\bigoplus_{j=0}^\infty H^i(\Lambda)$ is finite dimensional 
over $k$.
But this does not
yield a characterization, since the converse does not hold as shown in 
\cite{BGMS}.
For the convenience of the reader we will include their example. 
For $q\in k$ consider the finite dimensional algebra
$\Lambda_q=k\langle x,y\rangle/{\langle x^2,xy+qyx,y^2\rangle}$, where $k\langle x,y\rangle$
is the free algebra in two generators. It is easy to see that $\Lambda_q$ is a selfinjective algebra so 
$\gldim\Lambda_q=\infty$. If $q$ is not a root of unity, then it
is shown in \cite{BGMS} that $\text{dim}_k H^*(\Lambda)=5$.

\medskip

We will need the following observation:
Let $\Lambda=kQ/I$. We denote the primitive orthogonal idempotents of $\Lambda$ corresponding
to the vertices $\{1,\dots,n\}$ of $Q$ by $e_1,\dots,e_n$. Let $w=(i|\alpha_1,\dots,\alpha_r|j)$ be a path 
in $Q$, and assume that  $i\neq j$. Then $\bar w=e_i\bar w=e_i\bar w-\bar we_i$, since $\bar we_i
=0$ if $i\neq j$. Therefore $\bar w\in[\Lambda,
\Lambda]$ in this case.
Clearly, $e_i\notin [\Lambda,\Lambda]$, for $1\le i\le n$. If we assume in addition that $\Lambda$ has finite global dimension, we have more: In this case, since $H_0(\Lambda)$ is spanned as a vector space by the residue classes of $e_1,\ldots, e_n$, every proper 
cycle  $w$ of $Q$  has the property that its residue class $\bar w$ 
in $\Lambda$ belongs to
$[\Lambda,\Lambda]$.

\medskip

The following result was first shown in \cite{L} and reproved in \cite{I}. 
 This theorem is usually referred to as the 
no-loop conjecture. (Actually a stronger result is shown in \cite{L}. We
refer to this paper for details.)

\begin{theorem} Let $\Lambda=kQ/I$ be a finite dimensional algebra 
such that $Q$ 
contains a loop.
Then $\gldim\Lambda=\infty.$
\end{theorem}

\medskip

\noindent
\begin{proof} We give a sketch of the proof, and we use \ref{Hochschild}. Assume 
to the contrary that $\Lambda$ is of finite global dimension and let $\alpha$ 
be the loop in $Q$. 
Let $e_1,\dots,e_n\in\Lambda$ 
be the elements corresponding to the vertices of $Q$. By the previous remarks, the residue classes of
$\bar e_1,\dots,\bar e_n$ form a 
basis of $\Lambda/[\Lambda,\Lambda]$,
and $\bar\alpha\in[\Lambda,\Lambda]$. 
But it is
readily checked that this yields a contradiction, hence 
$\gldim\Lambda=\infty$.
\end{proof}

There is a local version of the no-loop conjecture called the strong no-loop 
conjecture. This was 
only proved recently in \cite{ILP} and is formulated as follows.

\begin{theorem} Let $\Lambda=kQ/I$ be a finite dimensional algebra 
such that $Q$ contains 
a loop at the vertex $i$.
Then the simple module $S(i)$ has infinite projective dimension.
\end{theorem}

A more general version of the strong no-loop conjecture is still open 
and is called the extension conjecture in \cite{ILP}.

\smallskip

\noindent
{\bf Extension conjecture} {\it Let $\Lambda=kQ/I$ be a finite 
dimensional algebra such 
that $Q$ contains a loop at the vertex $i$, then 
$\Ext^j_\Lambda(S(i),S(i))\neq 0$ for 
infinitely many $j$.}

\medskip

We point out that the more general question, whether 
$\Ext^1_\Lambda(S,S)\neq 0$ for a  $\Lambda$-module $S$ implies 
$\Ext^j_\Lambda(S,S)\neq 0$ for all $j$, has a negative answer. 
For this
consider the algebra $\Lambda$ given by the following quiver with 
relations (see also \cite{GSZ})

\begin{center}

\includegraphics[scale=1.1]{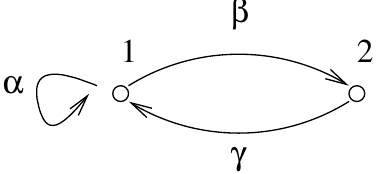}

\end{center}

\noindent
and relations $\alpha^2-\beta\gamma,\gamma\alpha\beta,\gamma\beta$.
Then the minimal projective resolution of $S(1)$ is easily computed as
$$\dots \to P(2)\to  P(1)\to P(1)\oplus P(2)\to P(1)\to S(1)\to 0$$

\noindent
and $\Omega^4S(1)=S(1)$. So $\Ext_\Lambda^j(S(1),S(1))=0$ for 
$j\equiv 3\, \text{mod}\,4$. But the extension conjecture clearly holds
in this example.

\medskip

We refer to \cite{GSZ} and \cite{LM} for some classes of algebras where the stronger version
of the extension conjecture will hold.

\section{Constructions}\label{constr}

In the previous section we have seen that one obstruction on $Q$ for
being the quiver of an algebra of finite global dimension was the existence of
a loop in $Q$. Here we will show that this is the only obstruction.
We are grateful to C. M. Ringel for pointing out to us that this construction
had already been carried out in \cite{DR1} and also in \cite{DR2} even in 
the more general setting of a species. We will come back to those papers also 
in the next  section.
For the convenience of the reader we give the construction of such an 
algebra of finite global dimension starting from an arbitrary quiver without
loops, but refer for a detailed proof to \cite{DR1}, \cite{DR2} or \cite{P}.
In general there will be several such algebras of finite global dimension
having the same quiver. This will be the subject of the next section.

\begin{theorem}\label{mono} Let $Q$ be a quiver without loops. 
Then there exists an admissible
ideal $I$ such that $\gldim\,kQ/I\le 2$.
\end{theorem}

\begin{proof}
We only give a sketch of the proof. Let $Q_0=\{1,\dots,n\}$ be the set of vertices of $Q$. For $2\le i\le n$ let $\alpha_{ij_i}$ 
be the arrows of $Q$ 
such that $e(\alpha_{ij_i})=i,$  $s(\alpha_{ij_i})<i$ and 
$\beta_{im_i}$ the arrows of $Q$ such
that $s(\beta_{im_i})=i$, $e(\beta_{im_i})<i$.
Let $I$ be the two sided
ideal in $kQ$ generated by $\alpha_{ij_i}\beta_{im_i}$ for
$2\le i\le n$ and all $j_i,m_i$.
Then $I$ is an admissible ideal and the algebra $\Lambda=kQ/$ is a finite dimensional monomial algebra. Using \cite{GHZ} one can easily show that
$\gldim\Lambda\le 2$.
\end{proof}

One can even show that the algebra constructed in the proof is 
strongly quasi-hereditary in the sense of 
Ringel. For a definition we refer to Section \ref{prelim}.
\medskip

We illustrate the previous theorem with a specific example. We will use the following notation.
The first set of arrows occurring in the sketch of the proof of \ref{mono}
is denoted by $\alpha_{i*}$ while the second is
denoted by $\beta_{i*}$ for a vertex $i$ of $Q$.  
Let $Q$ be the following quiver:

\begin{center}

\includegraphics[scale=1.0]{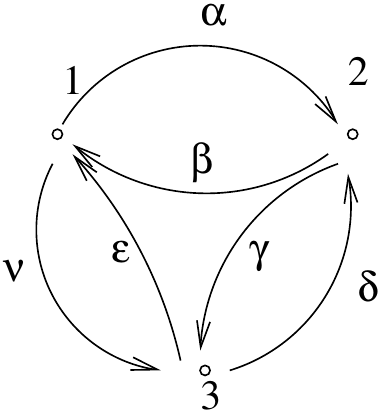}

\end{center}

\noindent      
Thus $\alpha_{2*}=\{\alpha\},\beta_{2*}=\{\beta\}$
and $\alpha_{3*}=\{\nu,\gamma\},\beta_{3*}=\{\delta,\epsilon\}$.

\noindent
Let $I=\langle\alpha\beta,\gamma\delta,\gamma\epsilon,
\nu\delta,\nu\epsilon\rangle$ and  $\Lambda = kQ/I$.

Then one can read the indecomposable projective
$\Lambda$-modules and their Loewy series of  from the following diagrams:

\medskip

\begin{center}

\includegraphics[scale=0.8]{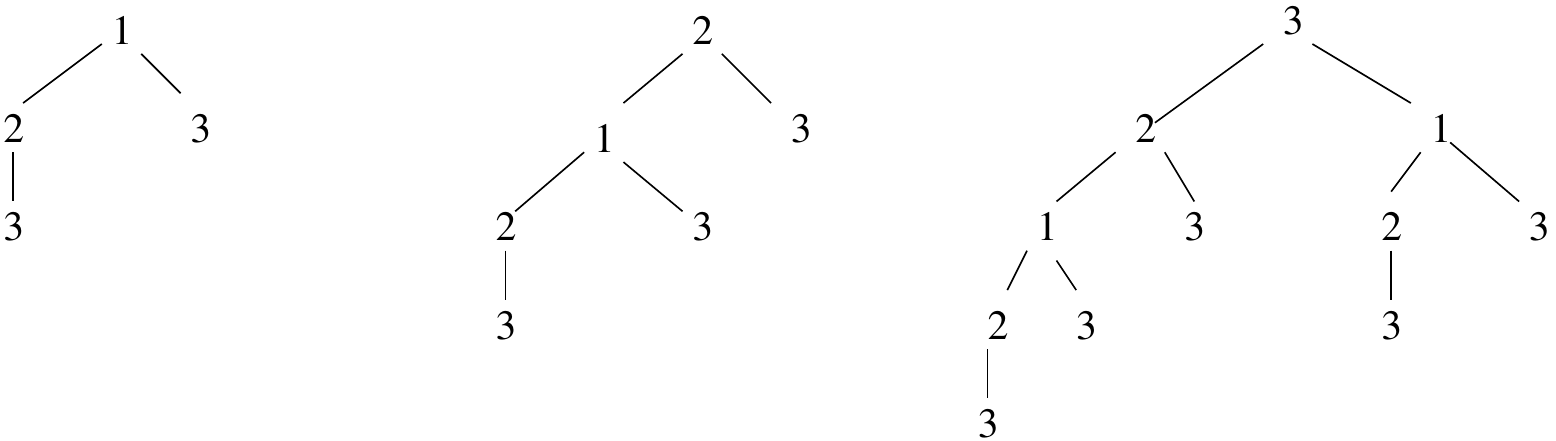}

\end{center}

The following exact sequences are the minimal projective 
resolutions of the simple $\Lambda$-modules

\begin{center}
$0\to P(1)^2\oplus P(2)\to P(2)\oplus P(3)\to P(1)\to S(1)\to 0$

\medskip

$0\to P(1)\oplus P(2)\to P(3)\oplus P(1)\to P(2)\to S(2)\to 0$

\medskip

$0\to P(1)\oplus P(2)\to P(3)\to S(3)\to 0.$
\end{center}

The standard modules are given by
$$\Delta(1)=P(1), \Delta(2)=P(2)/P(1)\, \text{and}\, \Delta(3)= S(3).$$

One may ask even more detailed questions, such as the existence of
an admissible ideal $I$ in a given quiver $Q$ such that
the global dimension of $kQ/I$ is a prescribed 
natural number. Trivially one needs more conditions. 
For example if there is no path of length two, then 
$\gldim\,kQ/I\le 1$. Some answers are given in \cite{P}, but 
the precise nature of the needed conditions remains unclear to us. For example
the following can be shown (see for example \cite{P}).
 
\begin{theorem} Let $Q$ be a quiver without loops containing
a path of length $d$ that consists of $d$ pairwise distinct arrows. 
Then
there exists an admissible ideal $I$ generated by paths such that
$\gldim\,kQ/I=d$.
\end{theorem}

We point out that the converse is not true. Consider the algebra
$\Lambda$ given by the cyclic quiver $Q$ on three vertices $1,2,3$
and arrows $\alpha,\beta,\gamma.$  

\begin{center}

\includegraphics[scale=0.8]{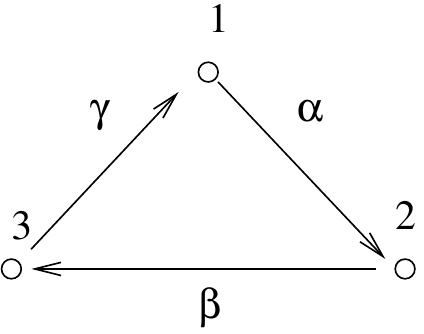}

\end{center}

\noindent
Choose as admissible ideal of $kQ$ the ideal $I=\langle\gamma\alpha\beta,
\alpha\beta\gamma\alpha\rangle$.
Then it is a straightforward computation to show that the following 
are minimal projective resolutions of the simple modules over $kQ/I$:
$$0\to P(3)\to P(2)\to P(2)\to P(1)\to S(1)\to 0$$
$$0\to P(3)\to P(2)\to S(2)\to 0$$
$$0\to P(3)\to P(2)\to P(3)\to P(1)\to P(3)\to S(3)\to 0.$$

\noindent
We have $\gldim\,kQ/I=4$, but there is no path in $Q$
consisting of four pairwise different arrows. However, we can say the following:

\begin{proposition}\label{triv}
Let $Q$ be a quiver without loops and $\Lambda=kQ/I$ be a finite 
dimensional algebra with $\gldim\,\Lambda =d<\infty.$ Then there 
exists a path of length $d$ in $Q.$
\end{proposition}
\begin{proof} Assume to the contrary that $Q$ does not contain a path of 
length $d.$ Then the length of all the paths in $Q$ is bounded by $d-1$.
In particular we conclude that $Q$ does not contain an oriented cycle,
hence is directed. But we have remarked before in Section \ref{prelim}
that then  $\gldim\,\Lambda\le d-1$, a contradiction. So $Q$ contains
a path of length $d$.
\end{proof} 

We present now
another result in the spirit of \ref{mono}. 

\begin{theorem}\label{gldim} Let $Q$ be a quiver without loops 
containing $n$ vertices. Then there exists a monomial algebra 
$\Lambda=kQ/I$ with $\gldim\,\Lambda\le n.$
\end{theorem}
\begin{proof} Let $Q_0=\{1,\dots,n\}$ be the set of vertices of $Q$. 
For each vertex $1\le i<n$ 
consider the following set of arrows:
$$\{\alpha_{ij}\,|\,s(\alpha_{ij})=i\,\,
\text{and}\,\,e(\alpha_{ij})>i\}$$

\noindent If $\beta$ is an arrow with $s(\beta)=e(\alpha_{ij})$ for some $\alpha_{ij},$
then $\alpha_{ij}\beta\in I.$ This defines $\Lambda$. It is easy to verify that
$\text{proj.dim}_\Lambda S(i)\le n-i+1$
\noindent for $1\le i\le n,$ so 
$\gldim\Lambda\le n.$
\end{proof}

We consider an example. Let $Q$ be the quiver from the example following
\ref{mono}. Then $I=\langle\alpha\gamma,\alpha\beta,\nu\delta,\nu\epsilon,
\gamma\delta,\gamma\epsilon\rangle$.
The indecomposable projective $\Lambda$-modules and their Loewy series can be read from the following diagrams:

\begin{center}

\includegraphics[scale=1.1]{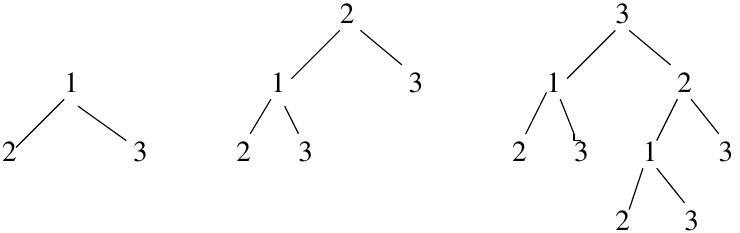}

\end{center}

\noindent The standard modules are again given by
$$\Delta(1)=P(1), \Delta(2)=P(2)/P(1)\, \text{and}\, \Delta(3)= S(3).$$
\noindent
Note that also in this case we obtain a strongly quasi-hereditary algebra.

\section{Bounds}\label{bounds}

Unless otherwise stated, $Q$ will denote in this section a quiver with no loops. We will deal with the 
following set of algebras:
$$\mathcal A(Q)=\{kQ/I\,|\,\text{dim}_k\,kQ/I<\infty\,\,\text{and}
\,\,\gldim\,kQ/I<\infty\}$$

\noindent By \ref{mono}
we know that $\mathcal A(Q)\neq \emptyset$.
We point out that $\mathcal A(Q)$ usually will contain an infinite
number of non-isomorphic algebras as the following example from \cite{Bo}
shows. Consider the following quiver $Q$

\begin{center}

\includegraphics[scale=1.0]{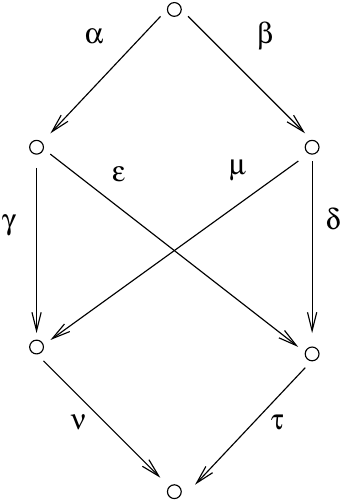}

\end{center}

\noindent
and $I_q=\langle\alpha\epsilon-\beta\delta,\alpha\gamma-\beta\mu,
\mu\nu-\delta\tau,\gamma\nu-q\epsilon\tau\rangle$ for some $q\in k$.
Set $\Lambda_q=kQ/I_q.$ Then $\gldim\Lambda_q\le 2$ and $\Lambda_q$
defines an infinite family of finite dimensional algebras, which all have
the same $k$-dimension.

\medskip

We then define the following 
$$d(Q)=\text{sup}\,\{\text{dim}_kkQ/I\,|\,kQ/I\in
\mathcal A(Q)\}$$

\noindent
and 
$$g(Q)=\text{sup}\,\{\gldim\,kQ/I\,|\,kQ/I\in
\mathcal A(Q)\}.$$

\noindent
The next result from \cite{S} yields a relationship between $d(Q)$ and 
$g(Q)$.
Its proof uses an upper semi-continuity argument on the algebraic variety of 
finite 
dimensional algebras of a fixed $k$-dimension.

\begin{theorem} 
There 
is a function $f\colon\mathbb N\to \mathbb N$ such $\gldim
\Lambda\le f(d)$ for all finite dimensional algebras $\Lambda$ with 
$\dm\Lambda\le d$ and $\gldim\Lambda < \infty.$
\end{theorem} 

\begin{corollary}\label{cor}
$g(Q)<\infty\,\,\text{if}\,\, d(Q)<\infty$.

\end{corollary}

We will address the following three questions.

\begin{questions}\label{Q} Let $Q$ be a quiver without loops.
\begin{itemize}

\item[{(1)}]  Is $d(Q)<\infty?$

\item[{(2)}]  Is the converse of \ref{cor}  true?

\item[{(3)}]  Is $\gldim\,kQ/I\le\dm kQ/I$ in case 
$\gldim\,kQ/I<\infty?$

\end{itemize}
\end{questions}

In general we do not have an answer to these questions. We will now survey 
what is known 
for special
classes of  algebras. But before, we would like to remind the reader of two classes 
of examples
which show that the values of the global dimension cannot only depend on 
the number of 
simple
modules nor on the Loewy length of the algebra. Recall that the Loewy 
length of $\Lambda$ 
is the least positive integer $d$ such that $\text{rad}^d\Lambda=0,$ where 
we have denoted by $\text{rad}\,\Lambda$
the radical of $\Lambda.$ The following two examples motivated us to consider $\mathcal A(Q)$ for a fixed quiver $Q$.

\medskip

\begin{example} The first example is due to E. Green \cite{Gr}, but see also 
\cite{H2}. It 
deals with the question 
of dependence of the global dimension on the number of simple modules.

\smallskip
\noindent
For each natural number $n$, let $Q_n$ be given by

\begin{center}

\includegraphics[scale=0.5]{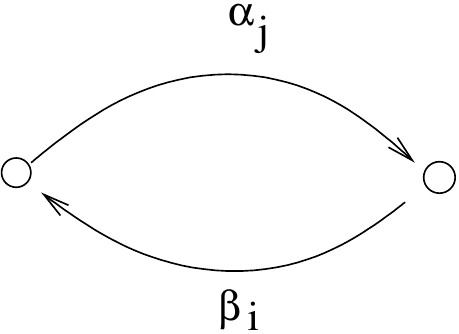}

\end{center}

\noindent
for $1\le i,j\le n$.

\noindent
Let $I_n=\langle\beta_i\alpha_j, 1\le i\le n, 1\le j\le i,\alpha_j\beta_i,
 2 \le j\le n,1 \le i \le j-1\rangle$.

\noindent
Set $\Lambda_n = kQ_n/I_n$. Then $\gldim
\Lambda_n =2n$.
Note that $\text{rad}^{2n}\Lambda_n\neq 0.$
\end{example}

\begin{example} The next example is due to E. Kirkman and J. Kuzmanovich 
and deals with the 
dependence of
the global dimension on the Loewy length. 

\medskip
\noindent Let $n\ge 1$ and let the quiver $Q_n$ be given as in the previous example.

\noindent
Set $I_n=\langle\beta_i\alpha_j\beta_l \,\,\text{for}\,\,
1 \le i,j,l\le n,\alpha_i\beta_{i+l}-\alpha_{i+l}\beta_{i+l}\,\,
\text{for}\,\,l\ge 1, \alpha_i\beta_j\\\noindent\,\,
\text{for}\,\,i>j,\beta_i\alpha_i
\,\,\text{for}\,\,1\le i\le n\rangle$.

\medskip

\noindent
Let $\Lambda_n = kQ_n/I_n$. Then $\gldim
\Lambda_n =2n+1$. Note that $\text{rad}^{4}\Lambda_n = 0.$
\end{example}

We will now mention some results where positive answers
to the questions above have been obtained for special classes of algebras.

\medskip

The first one is due to W. Gustafson \cite{G} and deals with Nakayama algebras.

\begin{theorem}\label{G}
Let $\Lambda$ be a Nakayama algebra of finite global dimension with $n$ non isomorphic simple modules. Then $\gldim\Lambda\le 2n-2$. Moreover, the Loewy length of $\Lambda$ 
is bounded by $2n-1$.
\end{theorem}
\begin{proof}  We give a proof of the bound of the global dimension
different from the one given in \cite{G} and refer for the bound of the 
Loewy length to \cite{G}. 
We will proceed
by induction on $n$. Trivially we may restrict to the case of a cyclic quiver. 
For $n=2$ it is easy to see that 
the only Nakayama algebra with two simple modules is given by the 
Auslander algebra of $k[t]/{\langle t^2\rangle},$
hence has global dimension two. So the result holds for $n=2$. 
Since $\Lambda$ has finite 
global dimension, there is a simple $\Lambda$-module, say $S(n)$ of projective dimension 1 (see \cite{Z2}).  Let $P=\bigoplus_{i=1}^{n-1}P(i)$ and let 
$\Gamma=\text{End}_\Lambda P$. Then
clearly $\Gamma$ is again a Nakayama algebra and by \cite{Z1},
$\gldim\Gamma\le\gldim\Lambda<\infty$. So by induction 
we have that 
$\gldim\Gamma\le2(n-1)-2$, since $\Gamma$ has $n-1$ simple modules.  The simple $\Gamma$-modules are all of the form $\Hom_{\Lambda}(P,S)$ where $S$ is a simple $\Lambda$-module such that $S$ is not isomorphic to $S(n)$. If $S\not\cong S(n)$ is such a simple $\Lambda$-module and
$$0\to P^t\to P^{t-1}\to\cdots\to P^1\to P(S)\to S\to 0$$
\noindent
is a minimal projective resolution of $S$, then all $P^i$ are indecomposable. Moreover
 $$(*)\,\,0\to \text{Hom}(P,P^t)\to\cdots\to\text{Hom}(P,P(S))
\to \text{Hom}(P,S)\to 0$$

\noindent
is a projective resolution of the simple $\Gamma$-module $\text{Hom}(P,S)$. The resolution $(*)$ is usually not minimal. This happens
if $\text{rad}\,P(n)\to P(n)$ occurs in $(*)$ and the map is the standard 
embedding. But this map can
only occur in $(*)$ at the end of the resolution, so we can conclude that 
$\gldim\Gamma\ge\gldim\Lambda -2$. So by induction 
$\gldim\Lambda\le 2n-2$.
\end{proof}

The result \ref{G} may be reformulated as follows. Let $Q$  be an oriented cycle with $n\ge 2$ vertices. Then $g(Q)\le 2n-2$. 

\medskip
We include now an example from \cite{G} showing that the above bound is optimal, that is
$g(Q)= 2n-2$ for $Q$ an oriented cycle with $n\ge 2$ vertices.
Let $Q$ be an oriented cycle with $n$ vertices and $n$ arrows $\alpha_i$ for
$1\le i\le n$ such that $s(\alpha_i)=i=e(\alpha_{i-1})$ for $1< i < n$ and
$e(\alpha_n)=1=s(\alpha_1)$. \begin{center}

\includegraphics[scale=0.8]{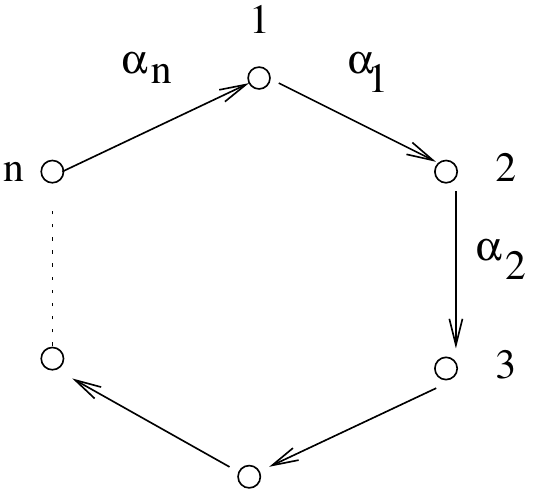}

\end{center}

\medskip

\noindent
Let $w_1=\alpha_1\dots\alpha_n$ and for 
$2\le i\le n-1$ let 
$w_i=\alpha_i\alpha_{i+1}\dots\alpha_n\alpha_1\dots\alpha_i.$ Then let
$\Lambda=kQ/{\langle w_1,\ldots,w_{ n-1}\rangle}$. Then $\gldim\Lambda=2n-2$. 
In fact,  
it is readily checked
that $\text{proj.dim}_\Lambda S(n-i)=2i+1$ for $0\le i <n-1$ and 
$\text{proj.dim}_\Lambda S(1)=2n-2$. 

\medskip

We point out that the second bound in \ref{G}
is also optimal. For more details we refer to 
\cite{G}.

\medskip
Contrary to the general situation, the global dimension for quasi-hereditary 
algebras is bounded
by a function on the number of simple modules. We have the following 
result from \cite{DR3}.

\begin{theorem}\label{DR}
If $\Lambda$ is a quasi-hereditary algebra with $n$ non isomorphic simple modules, 
then $\gldim\Lambda\le 2n-2.$
\end{theorem}

Results in \cite{DR1} and \cite{DR2} may also be reformulated. Let 
$$\mathcal A'(Q)=\{kQ/I\,|\,\dm kQ/I<\infty\,\,\text{and}
\,\,kQ/I \,\,\text{is quasi-hereditary}\}.$$ 
\noindent and 
$$d'(Q)=\sup\,\{\dm kQ/I\,|\,kQ/I\in
\mathcal A'(Q)\}.$$

Note that a quasi-hereditary algebra is always of finite global dimension (see for 
example \cite{DR3}), so $\mathcal A'(Q)\subset\mathcal A(Q)$.

\begin{theorem}\label{dr1}
$d'(Q)<\infty.$
\end{theorem}

\begin{corollary} Let $Q$ be a quiver without loops. Then
$$\sup\,\{\dm kQ/I\,|\,\gldim\,kQ/I\le 2\}
\in\mathbb N.$$
\end{corollary}
\begin{proof} It is well-known that any algebra of global dimension at most two is
quasi-hereditary \cite{DR3}. So the result follows from \ref{dr1}.
\end{proof}

For monomial algebras we obtain an affirmative answer for our question 3 
from \cite{IZ1}. Actually a stronger statement is shown in \cite{IZ1}. We
refer to this article for further details.

\begin{theorem}\label{IZ}

Let $\Lambda$ be a finite dimensional monomial algebra of finite global dimension. Then 
$\gldim\Lambda\le\dm\Lambda.$ 
\end{theorem}

One may restrict question one  to subclasses of algebras,
for example monomial algebras. But even then this  is an open problem. 
The results and examples below show that question one is related to
Section \ref{obstr} in the sense that there are obstructions
on the admissible ideal $I$ of an algebra of finite global dimension
of the form $kQ/I$ for a quiver $Q$ without loops.

\medskip

In the case of a monomial algebra $\Lambda=kQ/I$ we can say the following.

\medskip

Let $w\in Q$ be an oriented cycle, say $w=\alpha_1\dots\alpha_t.$ Then for
$1\le i\le t$ we set the  ``rotated" cycle $w_i=\alpha_i\dots\alpha_t\alpha_1\dots\alpha_{i-1}$. 

\begin{proposition}\label{prop1} Let $\Lambda=kQ/I$ be a monomial algebra 
of finite global dimension.
Then for each oriented cycle $w$ of $Q$ of length $t$ there is 
$1\le i\le t$ such that $w_i\in I$.
\end{proposition}
\begin{proof} Let $w$ be an oriented cycle of $Q$ of length $t$.  Without loss of generality we may assume that $w$ is a nonzero path in $\Lambda$, that is 
$w\notin I$. Then by \ref{Hochschild}
we infer that $w\in[\Lambda,\Lambda]$ since the image of $w$ is  $0$ in $H_0(\Lambda)$. 
Thus there are elements
$u_i,v_i\in\Lambda$ for $1\le i\le m$ such that 
$w=\sum_{i=1}^m u_iv_i-v_iu_i$. Since $\Lambda$ is
monomial the set of nonzero paths (including the paths of length  $0$) 
in $\Lambda$ forms a $k$-basis of $\Lambda$ (compare \ref{mono1}), so we may also assume that the $u_i$ and $v_i$ are nonzero paths in $\Lambda$. It follows that $w=u_iv_i$ 
for some $i$ and that $v_iu_i\in I$. But $v_iu_i$  is  clearly
a rotation of $w$. 
\end{proof}

\begin{corollary}\label{cor1} Let $\Lambda=kQ/I$ be a monomial algebra 
of finite global dimension.
Then for each oriented cycle $w$ of $Q$ we have that $w^2\in I$.
\end{corollary}
\begin{proof} If $w\in Q$ is an oriented cycle, then $w^2$ 
contains any rotation as a subpath,
hence $w^2\in I$ by \ref{prop1}.
\end{proof}

We remark that \ref{cor1} does not imply an affirmative answer to question 
1 for monomial algebras. 
If $Q$ is the quiver

\begin{center}

\includegraphics[scale=0.8]{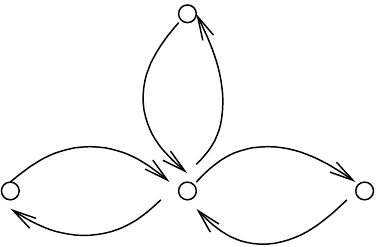}

\end{center}

\noindent
then there exist paths in $Q$ of arbitrary length which do not contain the 
square of a cycle in $Q$. This follows from \cite{Le}. See also \cite{Lo}.

\medskip

At the same time we do not know whether \ref{prop1} is sufficient for yielding an affirmative 
answer
to question 1 for the class of monomial algebras.  

\medskip

We would like to remark that \ref{cor1} will usually fail for non monomial 
algebras 
as the following example shows.
Let $\Lambda$ be given by the following quiver $Q$ \[
\xymatrix{\stackrel{2}{\circ}\ar@/^/^{\alpha}[rr]&&\stackrel{1}{\circ}\ar@/^/^{\delta}[ll]\ar@/^/^{\beta}[rr]&&
\stackrel{3}{\circ}\ar@/^/^{\gamma}[ll]
}\]




\noindent
and the ideal $I$ generated by $\delta\alpha-\beta\gamma$ and $\gamma\beta$.
Then the indecomposable projective $\Lambda$-modules and their Loewy series are 
given by the following diagrams:
\medskip

\begin{center}

\includegraphics[scale=0.9]{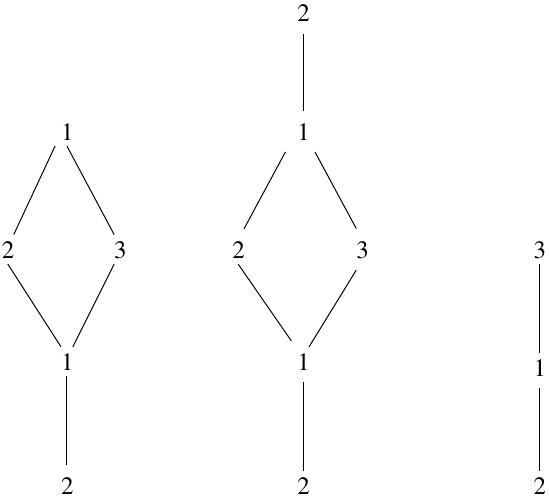}

\end{center}

\noindent
The minimal projective resolutions of the simple $\Lambda$-modules are:
$$0\to P(1)\to P(2)\oplus P(3)\to P(1)\to S(1)\to 0$$
$$0\to P(1)\to P(2)\to S(2)\to 0$$
$$0\to P(3)\to P(1)\to P(3)\to S(3)\to 0,$$

\noindent
hence $\text{gl.dim}\,\Lambda =2$. Consider the cycle $w=\alpha\delta$. Then  $w^2\notin I$. Note however, that the rotation $\delta\alpha$ of $w$ is in $I$.

\medskip

We end with a discussion of the following result by \cite{BHM}. First we recall 
the definition of a truncated 
$m$-cycle. Let $Q$ be a quiver, $I$ an admissible ideal and 
let $m\in\mathbb N$. A cycle $w\in Q$
is called a {\it truncated $m$-cycle}, if the composition of any $m$ composable 
arrows in $w$ belongs to $I$
and the composition of any $m-1$ composable arrows of $w$ does not belong 
to $I$. The following is
shown in \cite{BHM}.

\begin{theorem}\label{BHM} Let $\Lambda=kQ/I.$ If $Q$ contains a 
truncated $2$-cycle, then
$\gldim\Lambda=\infty$.
\end{theorem}

One of the questions asked in \cite{BHM} is whether the truncated $2$-cycle condition in \ref{BHM} could be replaced by a truncated $m$-cycle condition for
$m\ge 3$. The answer is negative,
as the following example shows. Indeed consider again the quiver $Q$ 
from the previous example.
Let $I=\langle\delta\alpha-\beta\gamma, \alpha\beta,\gamma\beta,\gamma\delta\rangle$ 
and $\Lambda=kQ/I$.
The indecomposable projective $\Lambda$-modules
are given as follows

\begin{center}

\includegraphics[scale=0.9]{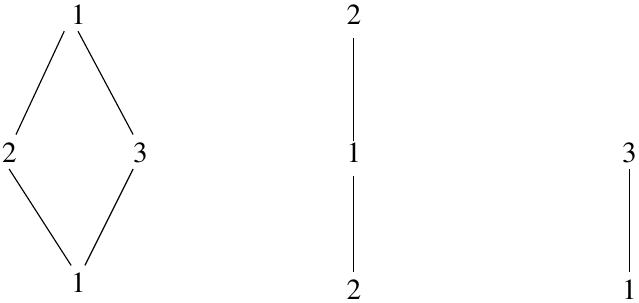}

\end{center}

\noindent
It is easy to see that $\gldim\Lambda =4$. In fact, the minimal 
projective resolutions of
the simple $\Lambda$-modules are given as follows
 $$0\to P(3)\to P(1)\to P(2)\oplus P(3)\to P(1)\to S(1)\to 0,$$
$$0\to P(3)\to P(1)\to P(2)\to S(2)\to 0$$
$$0\to P(3)\to P(1)\to P(2)\oplus P(3)\to P(1)\to P(3)\to S(3)\to 0.$$

\noindent
Now $\alpha\delta$ is a truncated $3$-cycle, since 
$\alpha\delta\alpha,\delta\alpha\delta\in I$ and
$\alpha\delta\notin I$ as well as $\delta\alpha\notin I.$

\end{document}